\documentclass[12pt]{amsproc}
\usepackage{amsmath,amsthm}
\usepackage{amsfonts,amssymb}
\usepackage{theoremref}
\usepackage{verbatim}
\usepackage[a4paper, margin=1.1in]{geometry}
\usepackage[shortalphabetic]{amsrefs}
\newtheorem{thm}{Theorem}[section]

\newtheorem{lem}[thm]{Lemma}
\newtheorem{prop}[thm]{Proposition}

\theoremstyle{remark}

\def\C{{\mathbb C}}
\def\H{{\mathbb H}}
\def\N{{\mathbb N}}
\def\R{{\mathbb R}}

\def\1{\text{\bf {1}}}

\DeclareMathOperator{\IM}{Im}

\begin{document}
\title[Riesz transforms]
{ Dimension free boundedness of \\
Riesz transforms for the Grushin operator}
\author{P. K. Sanjay}
 \author{ S. Thangavelu}
\address{Department of Mathematics\\ Indian Institute
of Science\\Bangalore-560 012. }
\email[Sanjay P. K.]{sanjay@math.iisc.ernet.in}
 \email[S. Thangavelu]{veluma@math.iisc.ernet.in}
 \address{Permanent address(Sanjay P. K.):Department of Mathematics, National Institute of
Technology, Calicut- 673 601.}
\keywords{Riesz Transforms, Grushin operator, Hermite operator, Transference method.}
\subjclass{42C, 42C05  , 43A65.}

\begin{abstract}Let $G = - \Delta_{\xi} - |\xi|^2 \frac{\partial^2}{\partial \eta^2}$ be the Grushin
operator on $\R^n \times \R.$ We prove that the Riesz transforms associated to this operator
are bounded on $L^p (\R^{n+1}), 1 < p < \infty$ and their norms are independent of the dimension
$n$. \end{abstract}
\maketitle
\section{Introduction}
\setcounter{equation}{0}
We  consider the Grushin operator $G = - \Delta_{\xi} - |\xi|^2 \frac{\partial^2}{\partial \eta^2}$ on $\R^n
\times \R$ which can be written formally as
$$Gf(\xi, \eta) = \frac{1}{2\pi} \int_{\R} e^{-i\lambda \eta} H(\lambda)f^{\lambda}(\xi) d\lambda $$ where
$$ H(\lambda) = - \Delta + \lambda^2 |\xi|^2$$ is the scaled Hermite operator on $\R^n$ and
$$ f^{\lambda}(\xi) = \int_{\R} f(\xi, \eta)e^{i\lambda \eta} d \eta $$ is the inverse
Fourier transform of $f$ in the $\eta$ variable.  In the light of this decomposition, we can define
the Riesz transforms associated to the operator $G$ as
  $$ R_jf(\xi,\eta)  = \int_{-\infty}^\infty  e^{-i\lambda \eta}
 R_j(\lambda)f^\lambda(\xi) d\lambda $$ and  $$ R_j^*f(\xi,\eta)  = \int_{-\infty}^\infty  e^{-i\lambda
\eta}
 R_j^*(\lambda)f^\lambda(\xi) d\lambda $$  for $ j=1,2,3,\ldots,n. $ Here, $$ R_j(\lambda) =
A_j(\lambda)H(\lambda)^{-\frac{1}{2}}, R_j^*(\lambda) =
 A_j(\lambda)^*H(\lambda)^{-\frac{1}{2}} $$ are the Riesz transforms associated to the Hermite
operator which has the decomposition $$ H(\lambda) = \frac{1}{2} \sum_{j=1}^{n}
\left(A_j(\lambda)A_j(\lambda)^*+ A_j(\lambda)^*A_j(\lambda)\right)$$
where
$$ A_j(\lambda) = -\frac{\partial}{\partial \xi_j}+\lambda \xi_j ,
 A_j(\lambda)^* = \frac{\partial}{\partial \xi_j}+\lambda \xi_j $$
are the creation and annihilation operators on $\R^n$.  The boundedness of these Riesz
transforms on $L^p(\R^{n+1})$ is known \cites{jotsaroop2011riesz, Baudoin18012012}. In
\cite{Baudoin18012012}, the authors have proved the boundedness of Riesz transforms associated to
a much larger class of smooth locally subelliptic diffusion operators on  smooth connected
non-compact manifold. In \cite{jotsaroop2011riesz} these Riesz transforms were
treated as operator valued Fourier multipliers for $L^p(\R^n)$ valued functions on $\R$ and the
boundedness was proved with the aid of a result of L. Weis \cite{MR1825406} on the operator
valued Fourier multipliers.    The aim of this paper is to prove the following theorem about the
dimension free boundedness of the vector of Riesz transforms $\mathcal{R}f$.  That is, we
consider the operator $\mathcal{R} = (R_1, R_2, \cdots, R_n, R_1^*, R_2^*, \cdots, R_n^*)$ with
$$|\mathcal{R}f(\xi,\eta)|= \left(\sum^n_{j=1}|R_j f(\xi,\eta)|^2+
\sum^n_{j=1}|R_j^* f(\xi,\eta)|^2\right)^{1/2}$$ and prove:
\begin{thm}\thlabel{main}  For each $1<p<\infty$, there exists a constant $C_p$ independent
of the dimension n, such that for all $f\in L^p(\R^n \times \R),$
$$ \left\| \mathcal{R}f \right\| _{p} \leq C_p
\| f \|_p.$$
\end{thm}
E. M. Stein  introduced the notion of the dimension free boundedness for the vector of Riesz
transforms in \cite{MR699317}.  Stein considered the Euclidean Riesz transforms and proved
the result using the technique of g-functions.  Later this technique was adapted for a
non-commutative situation by Harboure, et al. \cite{MR2047645} to
prove a similar result for the Riesz transforms associated to the Hermite operator.  Alternate
proofs for the Euclidean Riesz transforms were provided using the method of rotations by J.
Duoandikoetxea and J. L. Rubio de Francia \cite{MR780616} and using a transference argument
by  Gilles Pisier \cite{MR960544}.  Similar techniques have been used to prove the dimension
free boundedness of Riesz transforms in other contexts as well.
Coulhon, M{\"u}ller and Zienkiewicz \cite{MR1391221} proved the boundedness for the Riesz
transforms associated to the sub-Laplacian on the Heisenberg group.   In two works
\cite{MR2218202} and \cite{MR2091008} Francoise Lust-Piquard  provided alternate proofs for
the Riesz transforms associated to the Hermite operator and the Riesz transforms associated
to the sub-Laplacian.    Our proof follows a method very similar to that of
\cite{MR1391221},\cite{MR2091008} and \cite{MR2218202}.  The method of rotations applied to
the Euclidean Riesz transforms \cite{MR780616} involves  expressing the Riesz transforms as
the average of certain directional Hilbert transforms and using the boundedness (independent of
the direction) of these directional Hilbert transforms. An expression similar to this in the
Heisenberg group context is obtained in   \cite{MR1391221}.    In this paper we obtain a similar
representation for Riesz transforms associated to Grushin operators .

We also note that \thref{main} in turn implies that the Riesz transforms associated to the
Hermite operator satisfy dimension free bounds.  Since the Riesz transform for the Grushin
operator is an operator valued multiplier for the Fourier transform on $\R$, the corresponding
multiplier operator for the Fourier series is also bounded by the same norm.  This can be proved
using a generalisation of a transference result of  Karel de Leeuw \cite{MR0174937}. The proof is
similar to the proof in \cite{sanjay2011revisiting} of the boundedness of the Riesz transforms on
the reduced Heisenberg group using transference from the Heisenberg group.  We refer to
Theorem 2.1 and Theorem 2.2 of \cite{sanjay2011revisiting} for details.  Then the proof of the
Hermite Riesz transforms follows by looking at functions of the form $F(\xi, \eta) =
f(\xi)e^{ik\eta}$ on $\R^n \times [0, 2\pi).$
\section{Riesz transforms for the Grushin operator}
We first consider the individual Riesz transforms $R_j$ and $R_j^*$ for $j = 1, 2, \cdots , n$  and
show that they satisfy dimension free bounds on $L^p(\R^{n+1})$ for $ 1 < p < \infty.$  In order to
do this we introduce the operators $R_j^{\epsilon}$ and $R_j^{*\epsilon}$ which we will call the truncated Riesz transforms.  We only give details of  $R_j^{\epsilon}$ as the other one is similar.
Note that the Riesz transform $R_j(\lambda)$ associated to the Hermite operator $H(\lambda)$
can be written as $$R_j(\lambda) =A_j(\lambda)H(\lambda)^{-1/2} =
\frac{A_j(\lambda)}{\sqrt{\pi}}\int_{0}^{\infty}e^{-r H(\lambda)}r^{-1/2}dr.$$
Here $e^{-r H(\lambda)}$ is the Hermite semigroup.  For $\epsilon > 0,$ we define the truncated
Riesz transforms $R_j^{\epsilon}(\lambda)$ by $$R_j^{\epsilon}(\lambda) =
\frac{A_j(\lambda)}{\sqrt{\pi}}\int_{\epsilon^2}^{1/\epsilon^2}e^{-r H(\lambda)}r^{-1/2}dr.$$  Then the
truncated Riesz transforms $R_j^{\epsilon}$ for the Grushin operator are defined as
$$R_j^{\epsilon}f(\xi,\eta)  = \frac{1}{2 \pi} \int_{-\infty}^{\infty}  e^{-i\lambda \eta}
R_j^{\epsilon}(\lambda)f^\lambda(\xi) d\lambda. $$
We first prove
\begin{prop}\thlabel{trunc}For every $f \in L^2 (\R^{n+1}), R_j^{\epsilon}f \rightarrow R_jf$ in $L^2
(\R^{n+1})$ as $\epsilon \rightarrow 0.$
\end{prop}
\begin{proof}  It follows from the definition that
\[\int_{\R} \int_{\R^n}|R_j^{\epsilon}f(\xi, \eta) - R_jf(\xi, \eta)|^2 d\xi\, d\eta = 
\int_{\R} \left(\int_{\R^n}|R_j^{\epsilon}(\lambda)f^{\lambda}(\xi) - R_j(\lambda)f^{\lambda}(\xi)|^2
d\xi \right)d\lambda .
\]
The proposition follows once we show that for every $\lambda \in \R^*$
$$\int_{\R^n}|R_j^{\epsilon}(\lambda)f^{\lambda}(\xi) - R_j(\lambda)f^{\lambda}(\xi)|^2 d\xi
\rightarrow 0 \text{ as } \epsilon \rightarrow 0$$  and
$$\int_{\R^n}|R_j^{\epsilon}(\lambda)f^{\lambda}(\xi) - R_j(\lambda)f^{\lambda}(\xi)|^2 d\xi \leq 4
\int_{\R^{n}} |f^{\lambda}(\xi)|^2 d\xi. $$
To see these, we expand $f^{\lambda}$ in terms of scaled Hermite functions
$\Phi_{\alpha}^{\lambda}$ (see \cite{MR2008480} for definition) and use the fact that $$ A_j(\lambda) \Phi_{\alpha}^{\lambda} =
(2\alpha_j + 2) ^{1/2} |\lambda|^{1/2} \Phi_{\alpha+e_j}^{\lambda}.$$
where $e_j$ is the canonical unit vector of $\R^n$ with $1$ in the $j$\textsuperscript{th} entry and zero elsewhere. Then \[R_j^{\epsilon}(\lambda)f^{\lambda}(\xi) - R_j(\lambda)f^{\lambda}(\xi)=
   \sum_{\alpha \in \N^{n}}\frac{(2\alpha_j + 2) ^{1/2} |\lambda|^{1/2}}{\sqrt{\pi}}\left(\int_{A_\epsilon}e^{-(2|\alpha| + n)|\lambda| r
}r^{-1/2}dr  \right)(f^{\lambda},
\Phi_{\alpha}^{\lambda})  \Phi_{\alpha+e_j}^{\lambda}.
\]
where $A_{\epsilon} = {(0,\epsilon^2)\cup(1/\epsilon^2, \infty ) }.$
From this, it follows that
\[
\|R_j^{\epsilon}(\lambda)f^{\lambda} - R_j(\lambda)f^{\lambda}\|_2^2  =
\frac{1}{\pi}\sum_{\alpha \in \N^{n}} \frac{2\alpha_j + 2 }{2|\alpha| + n}|(f^{\lambda},
\Phi_{\alpha}^{\lambda}) |^2  \left(\int_{A_{\epsilon, \alpha} } e^{- r }r^{-1/2}dr\right)^2.
\]
where $A_{\epsilon, \alpha} = {(0,(2|\alpha| + n)|\lambda|\epsilon^2)\cup((2|\alpha| +
n)|\lambda|\epsilon^{-2}, \infty ) }.$ From the above equation it is clear that
$$\|R_j^{\epsilon}(\lambda)f^{\lambda} - R_j(\lambda)f^{\lambda}\|_2 \leq 2 \|f^{\lambda} \|_2$$
We also note that when $$ f^{\lambda}(\xi) = \sum_{(2|\alpha| + n)|\lambda| \leq N} (f^{\lambda},
\Phi_{\alpha}^{\lambda})\Phi_{\alpha}^{\lambda}(\xi)$$ is a finite linear combination of
$\Phi_{\alpha}^{\lambda},$
\[ \|R_j^{\epsilon}(\lambda)f^{\lambda} - R_j(\lambda)f^{\lambda}\|_2^2 \leq 
\frac{1}{\pi} \sum_{(2|\alpha| + n)|\lambda| \leq N} |(f^{\lambda},\Phi_{\alpha}^{\lambda})|^2
\left(\int_{(0,N\epsilon^2)\cup(N\epsilon^{-2}, \infty ) }\!\!\!\!\!\!\!\!\!\!\!\!\!\!\!\!\! e^{- r
}r^{-1/2}dr\right)^2 \]
which goes to 0   as $\epsilon \rightarrow 0.$  As such functions are dense in $L^2(\R^{n}),$ we
get $$\|R_j^{\epsilon}(\lambda)f^{\lambda} - R_j(\lambda)f^{\lambda}\|_2 \rightarrow 0, \text{ as }
\epsilon \rightarrow 0 $$ for any $ f \in L^2(\R^{n+1}).$
\end{proof}
For the individual Riesz transforms, we have the following result:
\begin{thm}\thlabel{singleriesz}For $j = 1, 2, 3, \cdots , n $ we have $$\|R_jf\|_p + \|R_j^*f\|_p \leq
C_p \|f\|_p, 1< p< \infty $$ for all $f \in L^p(\R^{n+1})$ where $C_p$ is independent of the
dimension.
\end{thm}
In order to prove this result, we claim that it is enough to prove
$$\|R_j^{\epsilon}f\|_p + \|R_j^{*\epsilon}f\|_p \leq C_p \|f\|_p $$
for the truncated Riesz transforms.  A proof of this  will be given in section \ref{sec:mainproof}
where we will show that $\left(R_j^{\epsilon}f \right)$ is Cauchy in $L^p(\R^{n+1})$ and hence
there exists an operator $S_j$, bounded on $L^p(\R^{n+1})$ such that $R_j^{\epsilon}f \rightarrow
S_jf$ as $\epsilon  \rightarrow 0.$  In view of \thref{trunc}, $R_j^{\epsilon}f \rightarrow R_jf$ for $f \in L^2(\R^{n+1})$ and hence $S_jf = R_jf$.  This will prove
the stated boundedness of $R_j$ ( and $R_j^*$). Now to prove the boundedness of $R_j^{\epsilon}$ and $R_j^{* \epsilon}$ we express these operators  as a superposition of certain
truncated Hilbert transforms.
\section{A representation for the truncated Riesz transforms }
The representation we obtain is very similar to the representation obtained in \cite{MR1391221}
for the Riesz transforms on the Heisenberg group $\H^{n}$.  Before stating this result, we recall
some definitions and notation.  For further details and proofs we refer to \cite{MR2008480}.

Recall that as a manifold $\H^n =\C^n \times \R$ and hence we write $(z,t), z = x+iy \in \C^n, t\in \R$ to denote
the elements of $\H^n$. The sub-Laplacian $\mathcal{L}$ on the Heisenberg group $\H^{n}$ can be written as the sum of the differential operators $ X_j={\left(\frac{\partial}{\partial x_j}+\frac{1}{2}y_j
\frac{\partial}{\partial t}\right)}$ and  $Y_j={\left(\frac{\partial}{\partial y_j}
-\frac{1}{2}x_j \frac{\partial}{\partial t} \right)}$ as $ \mathcal{L}=\sum^n_{j=1}(X_{j}^{2}+Y_{j}^{2}).$
The  sub-Laplacian is homogeneous of degree 2 with respect to the non-isotropic dilations
$\delta_r (z,t) =(rz, r^2t)$  of the Heisenberg group. Hence $p_s(z, t)$, the heat kernel associated to $\mathcal{L}$
satisfies $$p_{r^2s}(z, t) = r^{-(2n+2)}p_s(z/r, t/r^2).$$  This also follows from the following
explicit expression for $q_s(z,\lambda)$, the inverse Fourier transform  
of $p_s(z, t)$ in the $t$ variable,
$$q_s(z,\lambda)= \int_{\R} p_s(z, t)e^{i\lambda t}dt = (4 \pi)^{-n}\left(\frac{\lambda}{\sinh\lambda
s}\right)^{n}e^{-\frac{1}{4} \lambda \coth(s \lambda) |z|^2}.$$
By $\pi_{\lambda}$, we denote the Schr\"{o}dinger representation of the
Heisenberg group $\H^{n}$ acting on $L^2(\R^n)$ in the following manner
$$\pi_{\lambda}(x+iy, t)\phi(\xi) = e^{i\lambda t} e^{i\lambda (x.\xi+\frac{1}{2}x.y)}\phi(\xi + y).$$
When $t=0$, we will denote $\pi_{\lambda}(z, t)$ by $\pi_{\lambda}(z).$  $\pi_{\lambda}$ also
defines a representation of the group algebra $L^1(\H^{n})$ as $$ \pi_{\lambda}(f) =
\int_{\H^{n}}f(z,t)\pi_{\lambda}(z, t)dz\, dt = W_{\lambda}(f^{\lambda}) $$ where $W_{\lambda}(g) =
\int_{\C^n}g(z)\pi_{\lambda}(z)dz$ is the Weyl transform of $g$.
The representation given in \cite{MR1391221} for the Riesz transforms 
associated to the sub-Laplacian   $\mathcal{L}$ is
$$(X_i \mathcal{L}^{-\frac{1}{2}} f)(z,t) = -\frac{1}{4(2\pi)^{n+1}}\int_{\H^{n}}X_i p_1(w, s)H_{(w, s)}
f(z,t) dw \,ds$$
where $H_{(w, s)}$ is the Hilbert transform along a curve in the Heisenberg group.
For the Grushin operator, we obtain a similar representation involving certain operators
$T_{\epsilon}^{(z,t)}$ and certain differential operators $$\tilde{Z}_j = i\tilde{X}_j+\tilde{Y}_j; \,
\tilde{Z}^{*}_j = i\tilde{X}_j-\tilde{Y}_j$$ with $$ \tilde{X}_j= \left(\frac{\partial}{\partial x_j} -
\frac{y_j}{2}\frac{\partial}{\partial t}\right); \tilde{Y}_j= \left(\frac{\partial}{\partial y_j} +
\frac{x_j}{2}\frac{\partial}{\partial t}\right).$$
\begin{prop}\thlabel{rep} $$R_j^{\epsilon}f(\xi,\eta) =
\frac{1}{\sqrt{\pi}}\int_{\H^{n}}T_{\epsilon}^{(z,t)}f(\xi,\eta) \tilde{Z}_j p_1(z,t)dz\, dt$$
and $$R_j^{*\epsilon} f(\xi,\eta) = \frac{1}{\sqrt{\pi}}\int_{\H^{n}}T_{\epsilon}^{(z,t)}f(\xi,\eta)
\tilde{Z}^*_{j}p_1(z,t)dz \,dt$$
where
$$ T_{\epsilon}^{(z,t)}f(\xi,\eta) =  \int_{\epsilon < |r| < 1/\epsilon} f\left(\xi+ry,\eta+rx \cdot \xi+ r^2
\left(t+ \frac{x \cdot y}{2}\right)\right) \frac{dr}{r}$$ \end{prop}
We follow the method in \citeauthor{MR2218202} in proving this proposition.  First we prove the
following lemma which will be used in the proof of \thref{rep}.
\begin{lem} For any $\phi \in \mathcal{S}(\R^n)$,
\begin{itemize}
\item[($i$)] $ r \frac{\partial}{\partial \xi_j}\left( \pi_{\lambda}(rz)\phi(\xi)\right) = \left(\frac{\partial}{\partial y_j} + \frac{i
\lambda r^{2}x_{j}}{2}\right) \left(\pi_{\lambda}(rz) \phi(\xi)\right)$
\item[($ii$)]$ ir \lambda \xi_{j} \pi_{\lambda}(rz)\phi(\xi) = \left(\frac{\partial}{\partial x_j} - \frac{i \lambda
r^{2}y_{j}}{2}\right) \pi_{\lambda}(rz)\phi(\xi). $
\end{itemize}
\end{lem}
\begin{proof}We first look at the derivative of $\pi_{\lambda}(rz)$ with respect to the variable $y_j$  and see that
\begin{align*}
\frac{\partial}{\partial y_j} \big(\pi_{\lambda}(rz)\phi(\xi)\big) 
&=  \frac{\partial}{\partial y_j} \left(e^{i \lambda (rx
\cdot \xi + \frac{1}{2}r^2 x \cdot y)} \phi(\xi + ry)\right)\\
&=\frac{ i \lambda r^2 x_j}{2} \pi_{\lambda}(rz) \phi(\xi ) +  e^{i \lambda (rx
\cdot \xi + \frac{1}{2}r^2 x \cdot y)} \frac{\partial}{\partial y_j}\phi(\xi + ry)\\
&= \frac{ i \lambda r^2 x_j}{2} \pi_{\lambda}(rz) \phi(\xi )+ r \pi_{\lambda}(rz) \frac{\partial
\phi}{\partial \xi_j} (\xi). \end{align*}
Writing $$
\pi_{\lambda}(rz) \frac{\partial \phi}{\partial \xi_j} (\xi) = \left(\pi_{\lambda}(rz) \frac{\partial }{\partial
\xi_j} - \frac{\partial }{\partial \xi_j}\pi_{\lambda}(rz)\right)\phi(\xi ) + \frac{\partial }{\partial
\xi_j}\pi_{\lambda}(rz)\phi(\xi ) $$
 and noting that $$
\frac{\partial }{\partial \xi_j}\big(\pi_{\lambda}(rz)\phi(\xi)\big) =
 i \lambda r x_j \pi_{\lambda}(rz) \phi(\xi ) + \pi_{\lambda}(rz) \frac{\partial \phi}{\partial \xi_j} (\xi)
$$
we see that $$
\frac{\partial}{\partial y_j}\big( \pi_{\lambda}(rz) \phi(\xi)\big)  = 
\left( r \frac{\partial}{\partial \xi_j} - \frac{ i
\lambda r^2 x_j}{2} \right) \big(\pi_{\lambda}(rz)  \phi(\xi )\big) $$
which proves part (i). Part (ii) of the lemma follows from
$$ \frac{\partial}{\partial x_j} \big(\pi_{\lambda}(rz)\phi(\xi)\big) 
= i \lambda r \xi_j \pi_{\lambda}(rz) \phi(\xi )
+ \frac{ i \lambda r^2 y_j}{2}\pi_{\lambda}(rz) \phi(\xi )$$
\end{proof}
\begin{proof}[Proof of \thref{rep}]
The heat kernels for the sub-Laplacian and the Hermite operator are 
related via the group Fourier
transform on the Heisenberg group as follows:  $$  \hat{p_{s}}(\lambda) =
\int_{\H^{n}}p_{s}(z,t)\pi_{\lambda}(z,t)dz\,dt=e^{-s H(\lambda)}$$ 
We refer to \cite{MR2008480} for a proof of this.
Using the homogeneity of the heat kernel we get
\begin{eqnarray*}e^{-r^{2}H(\lambda)} &=& r^{-(2n+2)} \int_{\H^{n}}  \pi_{\lambda}(z,
t)p_{1}\left(\frac{z}{r},\frac{t}{r^2}\right) dz \,dt  \\
&=& r^{-2n} \int_{\C^n} \pi_{\lambda}(z)q_{1}\left(\frac{z}{r},- \lambda r^{2}\right) dz
=\int_{\C^n}  \pi_{\lambda}(rz)q_{1}\left(z,- \lambda r^{2}\right) dz.
\end{eqnarray*}
Since $R_j^{\epsilon}(\lambda) = \frac{1}{\sqrt{\pi}}A_j(\lambda)\int_{\epsilon < |r| <
1/\epsilon}e^{-r^{2}H(\lambda)}dr,$
we see that
$$
R_j^{\epsilon}(\lambda) = \frac{1}{\sqrt{\pi}}\int_{\epsilon < |r| <
1/\epsilon}\int_{\C^n}\left(-\frac{\partial}{\partial \xi_j}+\lambda \xi_j
\right)\pi_{\lambda}(rz)q_{1}\left(z,- \lambda r^{2}\right) dz \,dr;$$
$$R_j^{*\epsilon}(\lambda) =  \frac{1}{\sqrt{\pi}}\int_{\epsilon < |r| <
1/\epsilon}\int_{\C^n}\left(\frac{\partial}{\partial \xi_j}+\lambda
\xi_j\right)\pi_{\lambda}(rz)q_{1}\left(z,- \lambda r^{2}\right) dz\, dr.
$$
Now using the previous lemma,
\begin{eqnarray*}
\lefteqn{ir  \int_{\C^n}  \lambda \xi_{j} \pi_{\lambda}(rz)q_{1}\left(\frac{z}{r},- \lambda r^{2}\right) dz
} \\
&=& \int_{\C^n} \left( \frac{\partial}{\partial x_j} - \frac{i \lambda r^{2}y_{j}}{2}  \right) \pi_{\lambda}(rz) q_{1}(z,- \lambda r^{2}) dz \\
&=& -\int_{\C^n}  \pi_{\lambda}(rz)\left(\frac{\partial}{\partial x_j} + \frac{i \lambda
r^{2}y_{j}}{2}\right)q_{1}(z,- \lambda r^{2}) dz \\
&=& -\int_{\C^n}  \pi_{\lambda}(rz)\left(\frac{\partial}{\partial x_j} + \frac{i \lambda
r^{2}y_{j}}{2}\right)\left(\int_{\R}p_1(z,t)e^{i\lambda r^{2}t} dt \right) dz \\
&=& -\int_{\C^n}  \pi_{\lambda}(rz)\left(\int_{\R}\left(\frac{\partial}{\partial x_j} - \frac{y_{j}}{2}
\frac{\partial}{\partial t}\right)p_1(z,t)e^{i\lambda r^{2}t} dt \right) dz \\
&=& -\int_{\H^{n}}  \pi_{\lambda}(rz)\tilde{X}_{j}p_1(z,t)e^{i\lambda r^{2}t}  dz\,dt. \\
\end{eqnarray*}
Similarly,
\begin{eqnarray*}
\lefteqn{r \int_{\C^n}  \left(\frac{\partial}{\partial \xi_j} 
\pi_{\lambda}(rz) \right) q_{1}\left(\frac{z}{r},- \lambda
r^{2}\right) dz } \\
&=& \int_{\C^n} \left( \frac{\partial}{\partial y_j} + \frac{i \lambda r^{2}x_{j}}{2}  \right)\pi_{\lambda}(rz) q_{1}(z,- \lambda r^{2}) dz \\
&=& -\int_{\H^{n}}  \pi_{\lambda}(rz)\tilde{Y}_{j}p_1(z,t)e^{i\lambda r^{2}t}   dz\,dt. \\
\end{eqnarray*} Hence $$ R_j^{\epsilon}(\lambda) = \frac{1}{\sqrt{\pi}}\int_{\epsilon < |r| <
1/\epsilon}\int_{\H^{n}}\pi_{\lambda}(rz)\tilde{Z}_{j}p_1(z,t)e^{i\lambda r^{2}t}  dz\, dt\frac{dr}{r}$$ and
$$R_j^{*\epsilon}(\lambda) = \frac{1}{\sqrt{\pi}}\int_{\epsilon < |r| <
1/\epsilon}\int_{\H^{n}}\pi_{\lambda}(rz)\tilde{Z}^{*}_{j}p_1(z,t)e^{i\lambda r^{2}t}  dz \,dt \frac{dr}{r} .$$
Thus $$R_j^{\epsilon} f(\xi, \eta )= \frac{1}{\sqrt{\pi}}\int_{\R}\left(\int_{\epsilon < r <
1/\epsilon}\int_{\H^{n}}\pi_{\lambda}(rz)f^{\lambda}(\xi) \tilde{Z}_{j}p_1(z,t)e^{i\lambda r^{2}t}  dz\,
dt\frac{dr}{r}\right) e^{i \lambda \eta} d \lambda .$$
Now the proposition follows from the fact that $$\int_{\R}\pi_{\lambda}(rz)f^{\lambda}(\xi)
e^{i\lambda r^{2}t}  e^{i \lambda \eta} d \lambda = f\left(\xi+ry,\eta+rx \cdot \xi+ r^2 \left(t+ \frac{x
\cdot y}{2}\right)\right). $$
\end{proof}
\section{A transference result}
In the previous section, we obtained a representation for the Riesz transforms as the
superposition of certain operators $T_{\epsilon}^{(z,t)}$.  To prove the boundedness of the vector
of Riesz transforms, using the method of rotations, we need to prove that these operators are
bounded uniformly in $(z,t).$
\begin{prop}\thlabel{T} For $1 < p < \infty, $ there exists a constant $C_p$ independent of
$(z,t)\in \H^{n}, \epsilon > 0$ and the dimension $n$ such that, $$ \left\| T_{\epsilon}^{(z,t)}f
\right\|_{L^{p}(\R^{n+1})} \leq C_p \left\| f \right\|_{L^{p}(\R^{n+1})}.$$
\end{prop}
\begin{proof}
This  will be proved using Calder{\'o}n's method of transferring an operator on $L^p (X)$, for a
measure space $X$, to $L^p (G)$ when the group $G$ acts on $X$ by measure preserving
transformations.  The measure space $\R^{n+1}$, the group $\H^{n}$ and the group action
$$U(x+iy,t)(\xi, \eta) = (\xi -y,\eta-t+  \frac{x \cdot y}{2} - x . \xi) $$ are all the same as used
by Ratnakumar and Thangavelu in \cite{MR1612697} and we follow their notation.  Accordingly,
for $f \in L^p (\R^{n+1})$ and  $(\xi, \eta) \in \R^{n+1}$, we define the transferred function $F_{(\xi,
\eta)}$ on $\H^{n}$ as $$F_{(\xi, \eta)}(z,t) = f\left( U(z,t)(\xi, \eta)\right).$$  For $T \in
\mathcal{B}(L^p (\H^{n}))$, the transferred operator $T_0$ on $\mathcal{B}(L^p (\R^{n+1}))$ is
defined as $$T_0 f(\xi, \eta) = (T F_{(\xi, \eta)})(0).$$
For a curve $\gamma = \{ \gamma(t) \in \H^{n},t \in \R \}$ and a function $f$ on  $\H^{n}$, the Hilbert
transform of $f$ along $\gamma$ is defined as $$ H_{\gamma}f( w, s) = \int_{\R} f\left( ( w,
s)\gamma(r)^{-1} \right) \frac{dr}{r}. $$ When $\gamma$ is the curve $(rz, r^2 t)$, we will denote
$H_{\gamma}$ by $H_{(z,t)}$. By $H_{(z,t)}^{\epsilon}$ we denote the truncated Hilbert
transform where the integration is only over $\{r\in \R, \epsilon < |r| < 1/ \epsilon\}$.
When we transfer this operator $H_{(z,t)}^{\epsilon}$ to an operator on $L^p (\R^{n+1})$ under
the $U$ action defined above, we get
\begin{eqnarray*}
(H_{(z,t)}^{\epsilon})_0 f(\xi, \eta) &=&  H_{(z,t)}^{\epsilon}F_{(\xi, \eta)}(0) = \int_{\epsilon < |r| <
1/ \epsilon} F_{(\xi, \eta)}(-r z, -r^{2}t)\frac{dr}{r} \\
&=& \int_{\epsilon < |r| < 1/ \epsilon}\!\!\!\!f\left(\xi+ry,\eta+r(x \cdot \xi)+ r^2 \left(t+ \frac{x \cdot
y}{2}\right)\right) \frac{dr}{r}.\\
\end{eqnarray*}Hence we see that,
$(H_{(z,t)}^{\epsilon})_0 = T^{(z,t)}_{\epsilon}.$   Now our aim is to prove $$\|T^{(z,t)}_{\epsilon}f
\|_{L^p(\R^{n+1})} \leq C_p \|f \|_{L^p(\R^{n+1})} $$
This can be obtained using the transference method from the uniform boundedness of the
operator $H_{(z,t)}^{\epsilon}$.  To be precise, we use the fact that there exist a finite constant
$C_p$, independent of $(z,t) \in \H^{n}$, $n$ and $\epsilon > 0$ such that
$$\|H_{(z,t)}^{\epsilon}f\|_p \leq C_p \|f\|_p.$$  A proof of the above fact is indicated in
\cite{MR1391221}.  We will present it at the end of this section after proving the boundedness of
$T^{(z,t)}_{\epsilon}$.  Then
\begin{eqnarray*}
\int_{\R^{n+1}}\!\! |T^{(z,t)}_{\epsilon}f(\xi, \eta)|^p d\xi \,d\eta &=& \int_{\R^{n+1}}
|H_{(z,t)}^{\epsilon}F_{(\xi, \eta)}(0)|^p d\xi \,d\eta \\
&=& \int_{\R^{n+1}} |H_{(z,t)}^{\epsilon}F_{U(w, s)(\xi, \eta)}(0)|^p d\xi \,d\eta \\ \end{eqnarray*}
because of the invariance of the measure $d\xi \,d\eta$ under the $U$ action.  
Hence
\begin{align*} \int_{\R^{n+1}}\!\! |T^{(z,t)}_{\epsilon}f(\xi, \eta)|^p d\xi \,d\eta &=
 \frac{1}{|B_{R}(0)|} \int_{B_{R}(0)}\!\! \int_{\R^{n+1}}\!\!\!\!\!\! |H_{(z,t)}^{\epsilon}F_{U(w, s)(\xi,
\eta)}(0)|^p d\xi\, d\eta\, dw \,ds \\
&= \frac{1}{|B_{R}(0)|} \int_{B_{R}(0)}\!\! \int_{\R^{n+1}}\!\!\!\!\!\! |H_{(z,t)}^{\epsilon}F_{(\xi,
\eta)}(w, s)|^p d\xi \,d\eta\, dw \,ds \end{align*}
where $B_{R}(0)$ is the ball centred at origin in the Heisenberg group and radius $R$ under the
Koranyi norm $\|(w,s)\|= \left(|w|^4+|s|^2\right)^{1/4} .$
When  $(w, s) \in B_{R}(0),$ and $\epsilon <|r| < 1/ \epsilon,  (w, s) \left(\delta_{r}(z,t)\right)^{-1}
\in  B_{R+ \frac{h}{ \epsilon}}(0) $ where $h$ is the Koranyi norm of $(z, t).$ Hence, in the above
equality  $F_{(\xi, \eta)}$ can be replaced by $\tilde{F}_{(\xi, \eta)} = F_{(\xi, \eta)} \cdot
\chi_{B_{R+ h/ \epsilon}}(0)$. Now by an application of Fubini, we get
\begin{eqnarray*}\|T^{(z,t)}_{\epsilon}f \|_{p}^p &\leq& \frac{1}{|B_{R}(0)|} \int_{\R^{n+1}}\!
\int_{\H^{n}} \!\!\!\!\!\! |H_{(z,t)}^{\epsilon}\tilde{F}_{(\xi, \eta)}(w, s)|^p dw \,ds\, d\xi \,d\eta
\end{eqnarray*}
Now, from the uniform boundedness of the truncated Hilbert transforms, we get,
\begin{eqnarray*}
\|T^{(z,t)}_{\epsilon}f \|_{p}^p &\leq&  C_p \frac{1}{|B_{R}(0) |}  \int_{\R^{n+1}}\int_{\H^{n}}
|\tilde{F}_{(\xi, \eta)}(w, s)|^p  dw\,ds \,d\xi \,d\eta \\
 &=& C_p \frac{1}{|B_{R}(0)|} \int_{B_{R+h/\epsilon}(0)}  \int_{\R^{n+1}} 
|f(U(w, s)(\xi,\eta))|^p   d\xi d\eta dw ds \\
 &=& C_p \frac{|B_{R+h/\epsilon}(0)|}{|B_{R}(0)|}\|f\|_p^p = C_p \left( \frac{R+ h/
\epsilon}{R}\right)^{2n+2}\|f\|_p^p \\
\end{eqnarray*} again by the invariance $d\xi d\eta$ under the $U$ action.
Letting $ R \rightarrow \infty$, we see that $$\|T^{(z,t)}_{\epsilon}f \|_{L^p(\R^{n})} \leq C_p \|f
\|_{L^p(\R^{n})} $$
Coming back to the boundedness of $H_{(z,t)}^{\epsilon}$, we can use a technique in Lemma
3.1 of \cite{MR1101262}, to reduce this  to the boundedness of $H_{\gamma}^{\epsilon}$ on
$L^p(\R^2)$ for the curve $\gamma = \{(t, t^2), t \in \R\}$. Recall that
$$H_{(z,t)}f(w, s) = \int_{\R} f( w-rz, s- r^2 t - r\IM{(w.\bar{z})}) \frac{dr}{r}.$$  For $\sigma \in U(n),$
define $ \rho(\sigma)f(w, s) = f(\sigma w, s).$  Then
\begin{eqnarray*}
H_{(z,t)}(\rho(\sigma)f)(w, s) &=& \int_{\R} f( \sigma w-r\sigma z, s- r^2 t - r\IM{(w.\bar{z})})
\frac{dr}{r}.\\
\rho(\sigma^{-1})H_{(z,t)}(\rho(\sigma)f)(w, s) &=& H_{(z,t)}(\rho(\sigma)f)(\sigma^{-1} w, s)\\
&=& \int_{\R} f(  w-r\sigma z, s- r^2 t - r\IM{(\sigma^{-1}w.\bar{z})}) \frac{dr}{r}.\\
&=& \int_{\R} f(  w-r\sigma z, s-  t - r\IM{(w.\bar{\sigma z})}) \frac{dr}{r}.\\
\end{eqnarray*}
Since $\|\rho(\sigma)f\|_p = \|f\|_p$ and since there exists $\sigma \in U(n)$ such that
$\rho(\sigma)(x+iy)=x_1e_1 $, it is enough to consider the operator  $$ f \rightarrow \int_{\R} f( u_1 -rx_1,w', s-
r^2 t - rv_1 x_1) \frac{dr}{r}$$
which is equivalent to the  operator
$$ T_{(x, t, v)}f(u,s) =\int_{\R} f( u -rx,s- r^2 t - rv x) \frac{dr}{r}$$
acting on functions defined on $\R^2.$
For $\lambda_1, \lambda_2 > 0, $ let $$\delta_{\lambda_1, \lambda_2}f(u , s ) = f (\lambda_1 u,
\lambda_2 s). $$ Then $ \|\delta_{\lambda_1, \lambda_2}f\|_p = (\lambda_1  \lambda_2)^{-1/p}\|f\|_p
$ and
$$\delta_{1/\lambda_1, 1/\lambda_2}T_{(x, t, v)}\delta_{\lambda_1, \lambda_2}f(u,s) = \int_{\R} f(
u -r \lambda_1 x,s- r^2 \lambda_2 t - r\lambda_2 v x) \frac{dr}{r}$$
Since we can choose $\lambda_1, \lambda_2$ such that $\lambda_1 x = \lambda_2 t = 1 $  it is
enough to get uniform estimates for operators of the form $T_{(1, 1, a)}.$ That would imply
that  \\
\begin{align*} \|T_{(x, t, v)}f \|_p &= \|\delta_{\lambda_1, \lambda_2} T_{(1, 1,
vx/t)}\delta_{1/\lambda_1, 1/\lambda_2}f \|_p  \\
&=(\lambda_1  \lambda_2)^{-1/p}\|T_{(1, 1, vx/t)}\delta_{1/\lambda_1, 1/\lambda_2}f\|_p \\ &\leq
(\lambda_1  \lambda_2)^{-1/p}C_p \|\delta_{1/\lambda_1, 1/\lambda_2}f\|_p \\
&= (\lambda_1  \lambda_2)^{-1/p}C_p (\lambda_1  \lambda_2)^{1/p}\|f\|_p .\end{align*}
Now consider $ \tau_{a} f(u, s) = f ( u , s+au)$  so that $ \|\tau_{a} f\|_p = \|f\|_p. $
Since \begin{align*}
T_{(1, 1, a)} (\tau_{a}^{-1}f)(u, s) &= \int_{\R} (\tau_{a}^{-1}f)( u -r ,s- r^2  - ar) \frac{dr}{r} \\
&= \int_{\R} f( u -r ,s-au - r^2  ) \frac{dr}{r},
\intertext{we have}  \tau_{a} T_{(1, 1, a)} (\tau_{a}^{-1}f)(u, s) &= T_{(1, 1, a)} (\tau_{a}^{-1}f)(u,
s+au )\\
&=  \int_{\R} f( u -r ,s - r^2  ) \frac{dr}{r}.\\
\end{align*}
This is the Hilbert transform along the parabola $\gamma (r) = (r, r^2)$ in $\R^2$.  Since
\[
 \gamma (r) =
  \begin{cases}
   \delta_r (1, 1) & \text{if } r > 0 \\
   0 & \text{if }r = 0 \\
   \delta_{-r} (-1, 1)       & \text{if } r < 0
  \end{cases}
\]
and the linear space spanned by $\{\gamma (r)\}_{ r > 0}$  and  the linear space spanned by
$\{\gamma (r)\}_{ r < 0}$ are the same, namely $\R^2$, $\gamma (r)$ is a two-sided
homogeneous curve (see Def 3.1 in \cite{MR508453}*{p. 1261}).
Now we can appeal to Theorem 11  of \cite{MR508453}*{p. 1271} to get the uniform
boundedness of the  truncated Hilbert transforms  $H_{(z,t)}^{\epsilon}$ on $L^p(\H^{n})$.  That is,
there exist a finite constant $C_p$, independent of $(z,t) \in \H^{n}$, $n$ and $\epsilon > 0$ such
that $$\|H_{(z,t)}^{\epsilon}f\|_p \leq C_p \|f\|_p.$$
\end{proof}
\section{Proof of the main theorem}\label{sec:mainproof}
We first prove \thref{singleriesz} regarding the boundedness of the individual Riesz transforms
$R_j$ and $R_j^*.$
\begin{proof}[Proof of \thref{singleriesz}]
We first consider the operators $R_j^{\epsilon}$ and $R_j^{*\epsilon}.$  In the light of  \thref{rep}
and \thref{T}, we just need to prove that there exist a finite $C$, independent of $n$ such that
$$\|\tilde{Z}_j p_1(z,t)\|_{L^{1}(\H^{n})}, \|\tilde{Z}_j^* p_1(z,t)\|_{L^{1}(\H^{n})} \leq C.$$
As mentioned in Lemma 3 of \cite{MR1391221}, this follows from the fact that
$$p_r^n(z_1, z_2, \cdots, z_n, t) = p_r^1(z_1, \cdot)*p_r^1(z_2, \cdot),* \cdots,* p_r^1(z_n,
\cdot)(t)$$ where $p_r^n$ is the heat kernel on $\H^{n}.$ Now for the operators $R_j$ and
$R_j^{*},$ we need only to show that $\left(R_j^{\epsilon}f \right),$ $\left(R_j^{*\epsilon}f \right)$
are Cauchy in $L^p(\R^{n+1}).$
We have already seen that $\|R_j^{\epsilon}f\|_p \leq C \| T_{\epsilon}^{(z,t)}f \|_p $ and so it is
enough to prove that $T_{\epsilon}^{(z,t)}f$ is Cauchy in $L^p(\R^{n+1})$.  During the course of
the proof of \thref{T}, we had observed that the operator $T_{\epsilon}^{(z,t)}$ is obtained by
applying a transference on $H_{(z,t)}^{\epsilon}$ and that $$
\|T_{\epsilon}^{(z,t)}f\|_{L^p(\R^{n+1})} \leq \|H_{(z,t)}^{\epsilon}f\|_{L^p(\H^{n})}.$$
Being truncated Hilbert transforms $H_{(z,t)}^{\epsilon}f$ is Cauchy in $L^p(\H^{n})$ \cite{MR508453}*{Theorem 11}, and consequently, so is $T_{\epsilon}^{(z,t)}f$.
\end{proof}
Now we complete the proof of our main theorem.  As mentioned during the discussion of
\thref{singleriesz}, we only need to prove the boundedness of  the operator $
\mathcal{R}^{\epsilon}$ obtained by replacing $R_j$ and $R_j^*$ by their truncated versions.  We
will be closely following \cite{MR2091008} in proving this and so will skip some details. Using the
property
$$\frac{1}{r}\frac{\partial p_1}{\partial r} = \frac{1}{x_j}\frac{\partial p_1}{\partial x_j} =
\frac{1}{y_j}\frac{\partial p_1}{\partial y_j} $$ where $r = |z|$, we can rewrite
\begin{multline*} R_j^{\epsilon}f(\xi,\eta) = \frac{1}{\sqrt{\pi}}\int_{\H^{n}}(i x_j +
y_j)T_{\epsilon}^{(z,t)}f(\xi,\eta) \frac{1}{r}\frac{\partial p_1}{\partial r}(z,t)dz dt \\+ \frac{1}{2
\sqrt{\pi}}\int_{\H^{n}}( x_j - y_j)T_{\epsilon}^{(z,t)}f(\xi,\eta) \frac{\partial p_1}{\partial t}(z,t)dz dt
\end{multline*}
and  \begin{multline*}
R_j^{*\epsilon}f(\xi,\eta) = \frac{1}{\sqrt{\pi}}\int_{\H^{n}}(i x_j - y_j)T_{\epsilon}^{(z,t)}f(\xi,\eta)
\frac{1}{r}\frac{\partial p_1}{\partial r}(z,t)dz dt \\- \frac{1}{2 \sqrt{\pi}}\int_{\H^{n}}( x_j +
y_j)T_{\epsilon}^{(z,t)}f(\xi,\eta) \frac{\partial p_1}{\partial t}(z,t)dz dt
\end{multline*}
For a fixed $(\xi, \eta)$ we can choose $\lambda_1, \lambda_2,\ldots, \lambda_{2n}$ such that
${\sum_{j=1}^{2n} |\lambda_j| ^{2}=1}$ and
$$|\mathcal{R}^{\epsilon}f(\xi,\eta)| = \sum_{j=1}^{n} \left(\lambda_j 
\overline{R_j^{\epsilon}f(\xi,\eta)} +
\lambda_{j+n}\overline{R^{*\epsilon}_j f(\xi,\eta)}\right) $$
Now using the triangle inequality, H\"{o}lder's inequality and also Lemma 2(a) of
\cite{MR2091008}, we get
\begin{multline*}
|\mathcal{R}^{\epsilon}f(\xi,\eta)| \leq C \|T_{\epsilon}^{(z,t)}f(\xi,\eta)\|_{L^p(\frac{1}{r}\frac{\partial
p_1}{\partial r}dz dt)} \|x_1\|_{L^{p'}(\frac{1}{r}\frac{\partial p_1}{\partial r}dz dt)} \\
+ C \|T_{\epsilon}^{(z,t)}f(\xi,\eta)\|_{L^p(\frac{\partial p_1}{\partial t}dz dt)}
\|x_1\|_{L^{p'}(\frac{\partial p_1}{\partial t}dz dt)}
\end{multline*}
where $C$ is a universal constant.  Now the theorem follows from \thref{T} and an application of
Minkowski inequality  along with the fact (proved in \cite{MR2091008})that
$$\int_{\H^n}|x_1|^p \frac{1}{r}\frac{\partial p_1}{\partial r} dz\, dt, \int_{\H^n}|x_1|^p\frac{\partial p_1}{\partial t}dz \,dt \leq  A_p, \;p \geq 0, $$ where $A_p$ is independent of the dimension.

\begin{center}
{\bf Acknowledgments}

\end{center}
The authors are thankful to the referee for pointing out some errors and typos. The work of the first author is supported by the All India Council for Technical Education (AICTE).  The work of the second author is supported  by J. C. Bose Fellowship from
the Department of Science and Technology (DST) and also by a grant from UGC
via DSA-SAP.
\def\cprime{$'$}


\begin{bibdiv}
\begin{biblist}
\bib{Baudoin18012012}{article}{
      author={Baudoin, F.},
      author={Garofalo, N.},
       title={A note on the boundedness of {R}iesz transform for some
  subelliptic operators},
        date={2012},
     journal={International Mathematics Research Notices (to appear); Arxiv
  preprint arXiv:1105.0467},
  url={http://imrn.oxfordjournals.org/content/early/2012/01/18/imrn.rnr271.abstract},
}
\bib{MR1391221}{article}{
      author={Coulhon, T.},
      author={M{\"u}ller, D.},
      author={Zienkiewicz, J.},
       title={About {R}iesz transforms on the {H}eisenberg groups},
        date={1996},
        ISSN={0025-5831},
     journal={Math. Ann.},
      volume={305},
      number={2},
       pages={369\ndash 379},
         url={http://dx.doi.org/10.1007/BF01444227},
      review={\MR{1391221 (97f:22015)}},
}
\bib{MR0174937}{article}{
      author={Leeuw, Karel~de},
       title={On {$L_{p}$} multipliers},
        date={1965},
        ISSN={0003-486X},
     journal={Ann. of Math. (2)},
      volume={81},
       pages={364\ndash 379},
      review={\MR{0174937 (30 \#5127)}},
}
\bib{MR780616}{article}{
      author={Duoandikoetxea, J.},
      author={{Rubio~de Francia}, Jos{\'e}~L.},
       title={Estimations ind\'ependantes de la dimension pour les
  transform\'ees de {R}iesz},
        date={1985},
        ISSN={0249-6291},
     journal={C. R. Acad. Sci. Paris S\'er. I Math.},
      volume={300},
      number={7},
       pages={193\ndash 196},
      review={\MR{780616 (86e:42028)}},
}
\bib{MR2047645}{article}{
      author={Harboure, E.},
      author={Rosa, L.~de},
      author={Segovia, C.},
      author={Torrea, J.~L.},
       title={{$L^p$}-dimension free boundedness for {R}iesz transforms
  associated to {H}ermite functions},
        date={2004},
        ISSN={0025-5831},
     journal={Math. Ann.},
      volume={328},
      number={4},
       pages={653\ndash 682},
         url={http://dx.doi.org/10.1007/s00208-003-0501-2},
      review={\MR{2047645 (2006a:42017)}},
}
\bib{jotsaroop2011riesz}{article}{
      author={Jotsaroop, K.},
      author={Sanjay, P~K},
      author={Thangavelu, S.},
       title={{R}iesz transforms and multipliers for the {G}rushin operator},
        date={2011},
     journal={Journal d'Analyse Mathematique (to appear); Arxiv preprint
  arXiv:1105.3227},
}
\bib{MR2091008}{article}{
      author={Lust-Piquard, F.},
       title={Riesz transforms on generalized {H}eisenberg groups and {R}iesz
  transforms associated to the {CCR} heat flow},
        date={2004},
        ISSN={0214-1493},
     journal={Publ. Mat.},
      volume={48},
      number={2},
       pages={309\ndash 333},
         url={http://dx.doi.org/10.5565/PUBLMAT_48204_02},
      review={\MR{2091008 (2005g:43014)}},
}
\bib{MR2218202}{article}{
      author={Lust-Piquard, F.},
       title={Dimension free estimates for {R}iesz transforms associated to the
  harmonic oscillator on {$\Bbb R^n$}},
        date={2006},
        ISSN={0926-2601},
     journal={Potential Anal.},
      volume={24},
      number={1},
       pages={47\ndash 62},
         url={http://dx.doi.org/10.1007/s11118-005-4389-1},
      review={\MR{2218202 (2006k:42012)}},
}
\bib{MR960544}{incollection}{
      author={Pisier, G.},
       title={Riesz transforms: a simpler analytic proof of {P}.-{A}. {M}eyer's
  inequality},
        date={1988},
   booktitle={S\'eminaire de {P}robabilit\'es, {XXII}},
      series={Lecture Notes in Math.},
      volume={1321},
   publisher={Springer},
     address={Berlin},
       pages={485\ndash 501},
         url={http://dx.doi.org/10.1007/BFb0084154},
      review={\MR{960544 (89m:60178)}},
}
\bib{MR1612697}{article}{
      author={Ratnakumar, P.~K.},
      author={Thangavelu, S.},
       title={Spherical means, wave equations, and {H}ermite-{L}aguerre
  expansions},
        date={1998},
        ISSN={0022-1236},
     journal={J. Funct. Anal.},
      volume={154},
      number={2},
       pages={253\ndash 290},
         url={http://dx.doi.org/10.1006/jfan.1997.3135},
      review={\MR{1612697 (99h:33047)}},
}
\bib{MR699317}{article}{
      author={Stein, E.~M.},
       title={Some results in harmonic analysis in {${\bf R}^{n}$}, for
  {$n\rightarrow \infty $}},
        date={1983},
        ISSN={0273-0979},
     journal={Bull. Amer. Math. Soc. (N.S.)},
      volume={9},
      number={1},
       pages={71\ndash 73},
         url={http://dx.doi.org/10.1090/S0273-0979-1983-15157-1},
      review={\MR{699317 (84g:42019)}},
}
\bib{MR1101262}{article}{
      author={Strichartz, R.~S.},
       title={{$L^p$} harmonic analysis and {R}adon transforms on the
  {H}eisenberg group},
        date={1991},
        ISSN={0022-1236},
     journal={J. Funct. Anal.},
      volume={96},
      number={2},
       pages={350\ndash 406},
         url={http://dx.doi.org/10.1016/0022-1236(91)90066-E},
      review={\MR{1101262 (92d:22015)}},
}
\bib{sanjay2011revisiting}{article}{
      author={Sanjay, P.~K},
      author={Thangavelu, S.},
       title={Revisiting {R}iesz transforms on {H}eisenberg groups},
        date={2012},
     journal={Revista Mathem\'atica Iberoamericana},
	  volume={28},
	  number={4},
    pages={1091\ndash 1108},
		eprint = {arXiv:1110.3236}
}
\bib{MR508453}{article}{
      author={Stein, E.~M.},
      author={Wainger, S.},
       title={Problems in harmonic analysis related to curvature},
        date={1978},
        ISSN={0002-9904},
     journal={Bull. Amer. Math. Soc.},
      volume={84},
      number={6},
       pages={1239\ndash 1295},
         url={http://dx.doi.org/10.1090/S0002-9904-1978-14554-6},
      review={\MR{508453 (80k:42023)}},
}
\bib{MR2008480}{book}{
      author={Thangavelu, S.},
       title={An introduction to the uncertainty principle},
      series={Progress in Mathematics},
   publisher={Birkh\"auser Boston Inc.},
     address={Boston, MA},
        date={2004},
      volume={217},
        ISBN={0-8176-4330-3},
        note={Hardy's theorem on Lie groups, With a foreword by Gerald B.
  Folland},
      review={\MR{2008480 (2004j:43007)}},
}
\bib{MR1825406}{article}{
      author={Weis, L.},
       title={Operator-valued {F}ourier multiplier theorems and maximal
  {$L_p$}-regularity},
        date={2001},
        ISSN={0025-5831},
     journal={Math. Ann.},
      volume={319},
      number={4},
       pages={735\ndash 758},
         url={http://dx.doi.org/10.1007/PL00004457},
      review={\MR{1825406 (2002c:42016)}},
}
\end{biblist}
\end{bibdiv}
\end{document}